\documentclass[11pt]{article}
\usepackage[english]{babel}
\usepackage[utf8]{inputenc} 
\usepackage[T1]{fontenc} 
\usepackage{lmodern} 
\usepackage{amsmath}
\usepackage{amssymb}
\usepackage{amsthm}
\usepackage{hyperref}
\usepackage{color}
\usepackage{graphicx}
\usepackage{titlesec}
\usepackage[font=footnotesize]{caption}
\usepackage[margin=3cm]{geometry}

\titleformat{\subsection}[runin]{\normalfont\itshape}{\thesubsection\hspace{8pt}}{3pt}{}[.] 

\newcommand\NN{\mathbb{N}} 
\newcommand\ZZ{\mathbb{Z}} 
\newcommand\RR{\mathbb{R}} 
\newcommand\EE{\mathbb{E}} 
\newcommand\PP{\mathbb{P}} 
\renewcommand\P{\mathcal{P}} 
\newcommand\Q{\mathcal{Q}} 
\newcommand\ind{\mathbf{1}} 
\newcommand\law{\mathcal{L}}
\newcommand\TV{\operatorname{TV}}
\renewcommand\mid{\,|\,} 
\newcommand\Ro{\mathcal{R}_0} 

\newtheorem{theorem}{Theorem}
\newtheorem{lemma}[theorem]{Lemma}

\newtheorem{proposition}[theorem]{Proposition}
\theoremstyle{definition}
\newtheorem{remark}[theorem]{Remark}
\newtheorem{example}[theorem]{Example}

\begin{document}

\title{SIR model with social gatherings}
\author{Roberto Cortez\footnote{Universidad Andres Bello, Departamento de Matem\'aticas, Sazié 2212, sexto piso, Santiago, Chile. E-mail: \texttt{roberto.cortez.m@unab.cl}. Supported by Iniciaci\'on Fondecyt Grant 11181082.}}

\maketitle

\begin{abstract}
	We introduce an extension to Kermack and McKendrick's classic susceptible-infected-recovered (SIR) model in epidemiology, whose underlying mechanism of infection consists of individuals attending randomly generated social gatherings. This gives rise to a system of ODEs where the force of infection term depends non-linearly on the proportion of infected individuals. Some specific instances yield models already studied in the literature, to which the present work provides a probabilistic foundation. The basic reproduction number is seen to depend quadratically on the average size of the gatherings, which may be helpful to understand how restrictions on social gatherings affect the spread of the disease. We rigorously justify our model by showing that the system of ODEs is the mean-field limit of the jump Markov process corresponding to the evolution of the disease in a finite population.
\end{abstract}


\section{Introduction}

\subsection{Classic SIR model}
\label{sec:classic_SIR}

The \emph{susceptible-infected-recovered (SIR) model}, introduced in 1927 by Kermack and McKendrick \cite{kermack-mckendrick1927}, is a simple system of ODEs representing the evolution of the spread of an infectious disease in a large population. It belongs to a broader class known as \emph{compartmental models in epidemiology}. More specifically:
\begin{equation}
\label{eq:classic_SIR}
\begin{split}
s' &= - \beta s i, \\
i' &= \beta s i - \gamma i, \\
r' &= \gamma i,
\end{split}
\end{equation}
where the prime denotes derivative with respect to time $t\geq 0$. Here, $s = s_t$, $i = i_t$ and $r = r_t$ denote the proportion of the population that is \emph{susceptible} to the disease, \emph{infected}, and \emph{recovered} (or \emph{removed}), respectively; thus, $s+i+r = 1$. These labels are referred to as \emph{compartments}.

The rationale behind \eqref{eq:classic_SIR} is as follows. Informally, the size of the population is assumed to be infinite\footnote{In the literature, the total number of individuals is typically finite and denoted $N$, and one works with (continuous versions of) the actual number of individuals in each compartment instead of their proportions, that is, $S = Ns$, $I = Ni$ and $R = Nr$.
However, in the present article, the case $N<\infty$ will refer to a Markov process representing the random evolution of the number of individuals of each type in the finite population, whose mean-field limit as $N\to\infty$ yields \eqref{eq:classic_SIR}. Thus, we will reserve the use of $N$ and the capital letters $S$, $I$, and $R$ for that setting.}. Encounters between pairs of individuals occur randomly among the population, and a susceptible individual can acquire the disease only when interacting with an infectious. The incidence rate, i.e., the global rate at which infections occur, is thus proportional to $s$, $i$, and a parameter $\beta$, corresponding to the average number of contacts per person per unit time, multiplied by the probability of disease transmission. Once infected, an individual can then spread the disease to other susceptibles via the same mechanism. After some time, the individual recovers and gains permanent immunity, and he or she can no longer infect others. The global rate of recovery is thus proportional to $i$ and $\gamma$, where $1/\gamma$ is the mean duration of the infectious period.

In the last decades, the model \eqref{eq:classic_SIR} and its variants have been used extensively to predict the evolution of disease spread, see for instance \cite{andersson-britton2000,capasso-serio1978,hethcote2000,liu-hethcote-levin1987}. In particular, during the ongoing COVID-19 pandemic, the SIR model has been used as a mathematical tool to study the effect of \emph{non pharmaceutical interventions} (\emph{NPIs}). Those are actions taken by authorities and individuals, apart from medical measures such as getting vaccinated and taking medicine, which aim to help slow the spread of the disease; see \cite{perra2021} for an extensive biomedical review. Examples of NPIs include wearing face masks, frequent hand washing, social distancing, bans on social gathering, lockdowns, border closures, etc. Mathematically, the effects of social distancing and other NPIs are typically modelled by modifying the term $\beta i$ in \eqref{eq:classic_SIR}, known as \emph{force of infection}, either by making $\beta$ depend on time or some additional variable (e.g.\ space), or by considering a non-linear dependence on $i$; see for instance \cite{cabrera-cordovalepe-gutierrezjara-vogtgeisse2021,cotta-naveiracotta-magal2020,kolokolnikov-iron2021,wangping-etal2020}.


\subsection{SIR model with social gatherings}
\label{sec:extended_SIR}

To the best of our knowledge, something that is currently missing in the literature of SIR-like models, is the study of the effect of restrictions on social gatherings. In previous works, it is implicitly assumed that these restrictions are considered as part of the social distancing effect, or even as part of all the NPIs as a whole, without any detailed description. Motivated by this, especially in the context of the current COVID-19 pandemic, our main goal is to introduce an extension of the classic SIR model \eqref{eq:classic_SIR} that accounts for the way social gatherings take place in the population and study their effects on the spread of the disease. More importantly, and this the main novelty of the present work, we provide an explicit probabilistic interpretation for the microscopic (i.e., person-to-person) mechanism of infection, in terms of individuals attending these social gatherings.

We now describe our model. Note that one of the underlying assumptions behind \eqref{eq:classic_SIR} is that encounters always take place between two individuals. In our extended setting, encounters, which we call \emph{gatherings}, can have any number of individuals. Specifically, in the simplest case where the size of the gatherings is some fixed number $\theta \in \{0,1,2,\ldots\}$, our model is given by the following system of ODEs:
\begin{equation}
\label{eq:extended_SIR}
\begin{split}
s' &= - \mu \theta s \left(1 - (1-p i)^{\theta-1} \right), \\
i' &= \mu \theta s \left(1 - (1-p i)^{\theta-1} \right) - \gamma i, \\
r' &= \gamma i.
\end{split}
\end{equation}
Here, $\mu>0$ is the rate at which gatherings occur, and $p\in[0,1]$ is the probability of transmission of the disease. Note that when $\theta = 2$, we recover the usual SIR model \eqref{eq:classic_SIR} with $\beta = 2\mu p$. The qualitative behaviour of \eqref{eq:extended_SIR} is similar to \eqref{eq:classic_SIR}: initially, there is a possible increase in the proportion of infected individuals, then it reaches a maximum, and then it decays to 0; see Lemma \ref{lem:properties_sir}. This can be appreciated in Figure \ref{fig:gSIR}, where we display the numerical solution of \eqref{eq:extended_SIR} for some values of the parameters.

\begin{figure}[t!]
	\includegraphics[width=\textwidth]{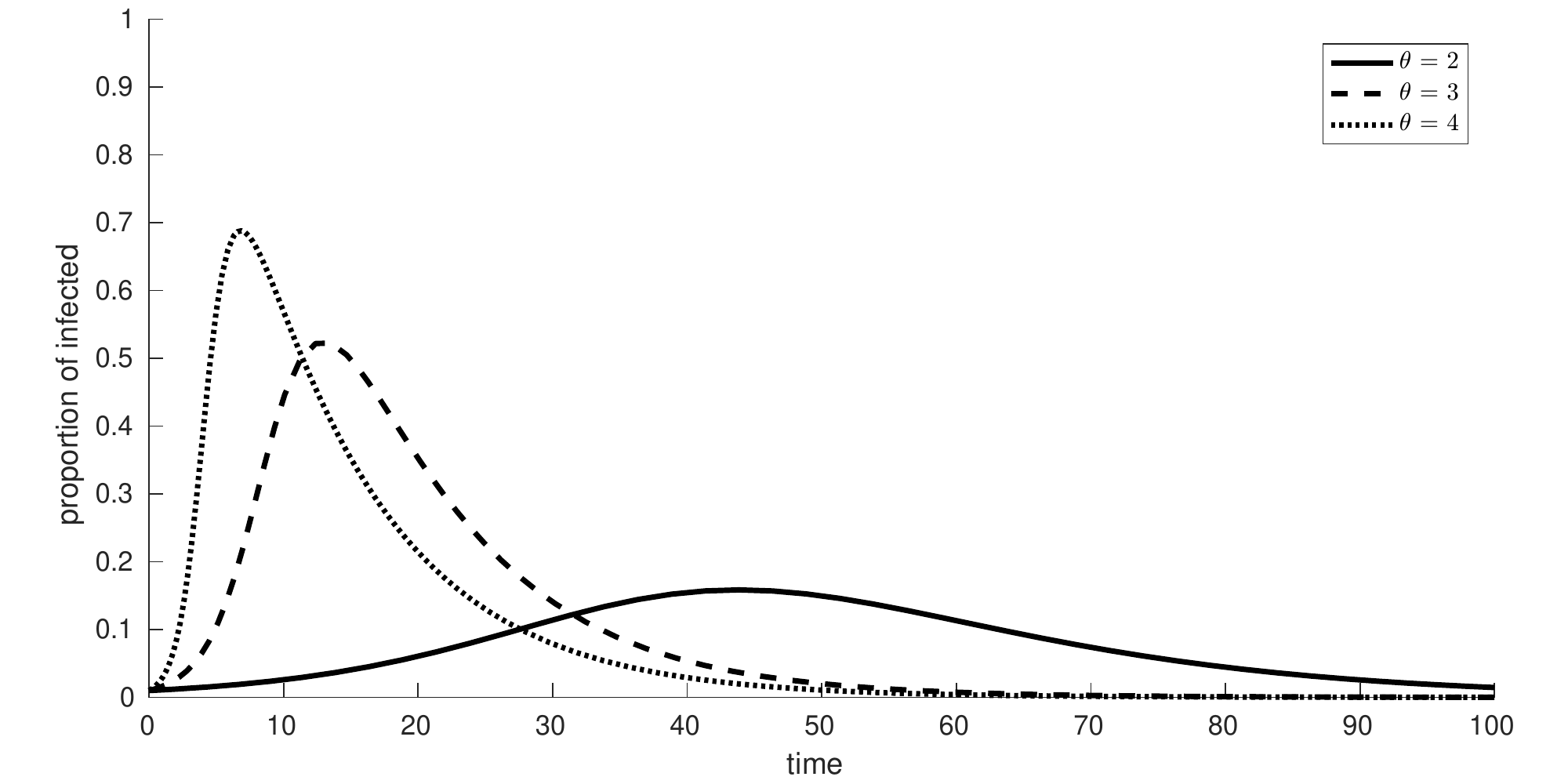}
	\caption{Evolution of the proportion of infected individuals $i$, obtained by solving numerically the SIR model with gatherings given by \eqref{eq:extended_SIR}. Parameters: $\mu = .5$, $p=.2$, $\gamma=.1$, $i_0 = .01$.}
	\label{fig:gSIR}
\end{figure}

The heuristic behind the proposed model is as follows:
\begin{itemize}
	\item An infected individual recovers at rate $\gamma$, as before.
	
	\item Gatherings occur at rate $\mu$ per individual.
	
	\begin{itemize}
		\item When a gathering takes place, randomly sample $\theta$ individuals from the infinite population. That is: independently, perform $\theta$ times the experiment of sampling from the set $\{\text{susceptible}, \text{infected}, \text{recovered}\}$ with respective probabilities $s$, $i$, and $r$.
		
		\item Each susceptible in the gathering will attempt to acquire the disease as many times as there are infectious individuals in the room, each time with probability $p$, independent of everything else. All those susceptibles that acquired the disease at least once, instantaneously and simultaneously become infected.
	\end{itemize}
	
\end{itemize}

The system \eqref{eq:extended_SIR} can be obtained from this description, see Section \ref{sec:heuristic_justification} for a more detailed but still informal derivation. A completely rigorous justification is provided in Section \ref{sec:LLN_SIR_model_gatherings}.

\subsection{Random size of gatherings}
\label{sec:random_size}
More generally, one can consider the case where the number of individuals attending each gathering is \emph{random}. That is, the sizes of the gatherings can be taken as independent copies of some random variable $\Theta$ on $\{0,1,2,\ldots\}$ with known distribution. Consequently, the incidence rate $\mu \theta s (1 - (1-p i)^{\theta-1} )$ in \eqref{eq:extended_SIR} (with $\Theta$ in place of $\theta$) is to be replaced by its expected value. Specifically, the system of ODEs is
\begin{equation}
\label{eq:extended_SIR_random}
\begin{split}
s' &= - \mu s B(i), \\
i' &= \mu s B(i) - \gamma i, \\
r' &= \gamma i,
\end{split}
\end{equation}
together with some initial condition $(s_0,i_0,r_0) \in [0,1]^3$ such that $s_0 + i_0 + r_0 = 1$, where the function $B:[0,1] \to \RR_+$ is defined as
\begin{equation}
\label{eq:Bi}
B(i) = \EE\left[ \Theta \left(1 - (1-p i)^{\Theta-1} \right) \right].
\end{equation}

In Section \ref{sec:properties} we study the main analytical properties of the system of ODEs \eqref{eq:extended_SIR_random}: well-posedness is provided in Lemma \ref{lem:properties_sir}, and in Lemma \ref{lem:comparison} we show that the disease spreads more slowly than in the classic SIR model \eqref{eq:classic_SIR}, provided that the basic reproduction number (recalled below) is the same.

Randomizing the size of the gatherings gives the model more flexibility. For instance, if authorities restrict social gatherings to have a maximum size $K$, then one can work with some $\Theta$ whose distribution is supported on $\{0,1,\ldots,K\}$ (e.g., discrete uniform). Moreover, if the distribution of $\Theta$ is such that the function $[0,1] \ni \xi \mapsto \EE[\xi^\Theta]$ has an explicit expression, then so does $\EE[\Theta \xi^{\Theta-1}] = \frac{d}{d\xi} \EE[\xi^\Theta]$; thus, the function $B(i)$ will have a closed form, leading to an explicit system of ODEs. This includes some frequently used distributions, such as the discrete uniform, geometric, negative binomial, etc. We highlight the following particular cases:

		
\begin{example}
When $\Theta \sim \text{binomial}(K,\alpha)$, it is easy to check that for all $\xi \in [0,1]$,
\begin{equation}
\label{eq:Exi^Theta}
\begin{split}
\EE[\xi^\Theta] &= \left(1 - \alpha (1-\xi) \right)^K
\quad \text{and}, \\
\EE[\Theta \xi^{\Theta-1}] &= K \alpha \left(1 - \alpha (1-\xi) \right)^{K-1},
\end{split}
\end{equation}
which gives $B(i) = K \alpha (1 - (1- \alpha p i)^{K-1} )$. The system of ODEs is (notice that when $\alpha = 1$, we recover \eqref{eq:extended_SIR}):
\begin{equation}
\label{eq:extended_SIR_binomial}
\begin{split}
s' &= - \mu K \alpha s \left(1 - (1- \alpha p i)^{K-1} \right), \\
i' &= \mu K \alpha s \left(1 - (1- \alpha p i)^{K-1} \right) - \gamma i, \\
r' &= \gamma i.
\end{split}
\end{equation}
\end{example}
	
\begin{example}
When $\Theta \sim \text{Poisson}(\lambda)$, it is straightforward to check that $\EE[\Theta \xi^{\Theta-1}] = \lambda e^{-\lambda (1-\xi)}$ for all $\xi \in [0,1]$, which then gives $B(i) = \lambda (1 - e^{-\lambda p i} )$. The system of ODEs is
\begin{equation}
\label{eq:extended_SIR_Poisson}
\begin{split}
s' &= - \mu \lambda s (1 - e^{-\lambda p i} ), \\
i' &= \mu \lambda s (1 - e^{-\lambda p i} ) - \gamma i, \\
r' &= \gamma i.
\end{split}
\end{equation}
Notice that an expression similar to $\mu \lambda s (1 - e^{-\lambda p i} )$ also appears in the model presented in \cite{kolokolnikov-iron2021}, although the underlying infection mechanism considered there is spatial in nature, and the formula with the exponential is obtained only after an approximation. The present article provides a different microscopic interpretation, by means of gatherings whose size follows a Poisson distribution, which justifies the use of a model like \eqref{eq:extended_SIR_Poisson}.
\end{example}

\begin{example}
	Consider the case where $\Theta$ has a \emph{logarithmic distribution} of parameter $\alpha \in (0,1)$, denoted $\Theta \sim \log(\alpha)$, that is,
	\[
	\PP(\Theta = k) = \frac{-1}{\log(1-\alpha)} \frac{\alpha^k}{k},
	\qquad \forall k = 1,2,\ldots.
	\]
	It is easy to check that $\EE[\Theta \xi^{\Theta-1}] = \frac{-1}{\log(1-\alpha)} \frac{\alpha}{1-\alpha \xi}$ for all $\xi \in [0,1]$, which, after a straightforward computation, gives
	\[
	B(i) = \frac{a i}{1 + b i},
	\qquad
	\text{for}
	\quad
	a = \frac{-\alpha^2 p}{(1-\alpha)^2 \log(1-\alpha)}
	\quad
	\text{and}
	\quad
	b = \frac{\alpha p}{1-\alpha}.
	\]
	The system of ODEs reads:
	\begin{equation}
	\label{eq:extended_SIR_logarithmic}
	\begin{split}
	s' &= - \frac{\mu a s i}{1+bi}, \\
	i' &= \frac{\mu a s i}{1+bi} - \gamma i, \\
	r' &= \gamma i.
	\end{split}
	\end{equation}
	This model, with incidence rate of the form $\frac{\mu a s i}{1+bi}$ for some constants $a,b$, was studied in \cite{capasso-serio1978} as the main example. One of its features is that it can be solved explicitly in the $s$-$i$ plane:
	\[
	i(s)
	= \begin{cases}
	-\frac{1}{b} + C s^\rho + \frac{s}{\rho - 1} & \rho \neq 1, \\
	-\frac{1}{b}  + Cs - s \log s & \rho = 1,
	\end{cases}
	\]
	where $\rho = \frac{\gamma b}{\mu a}$ and $C$ is a constant depending on the initial condition $(s_0, i_0)$; see \cite[Section 6]{capasso-serio1978} for more details. Again, the present article provides a probabilistic justification for the use of a model like \eqref{eq:extended_SIR_logarithmic}, by means of gatherings whose size follows a logarithmic distribution.
\end{example}

\begin{remark}
Our proposed mechanism of infection, by means of social gatherings, can be used to obtain variants of other compartmental models, such as SIRS, SIRV, SEIR, MSEIR, include birth and death, different distributions of the recovery time, etc.; see \cite{hethcote2000} for a review. To do so, one simply replaces $\beta s i$ (or the corresponding term representing the incidence rate) by $\mu s B(i)$, where $B(\cdot)$ is given by \eqref{eq:Bi}. For example, for the \emph{susceptible-exposed-infected-recovered (SEIR)} model, the corresponding extension with gatherings is
\[
\begin{split}
s' &= - \mu s B(i), \\
e' &= \mu s B(i) - \varepsilon e , \\
i' &= \varepsilon e - \gamma i, \\
r' &= \gamma i,
\end{split}
\]
where $e = e_t$ is the proportion of the population that has been exposed to the disease but hasn't become infectious yet, and the parameter $\varepsilon>0$ is the rate at which individuals change from exposed to infected. In the present article we study only the extension to the classic system \eqref{eq:classic_SIR}, because it is archetypal and arguably the most well-known model in this setting. Most of our developments can be easily adapted to other variants.
\end{remark}

\subsection{Reproduction number}
\label{sec:reproduction number}

The \emph{basic reproduction number} $\Ro$ is a crucial quantity in many epidemiological models. It is defined as the average number of new infections produced by an infected individual in a large population where almost everyone is susceptible to the disease. Its importance comes from the fact that the initial behaviour of the system (i.e., initial increase or decline of the small infected population) depends on whether the basic reproduction number exceeds the threshold value of 1. In other words, a major epidemic outbreak is possible if and only if $\Ro>1$.

Mathematically, $\Ro$ is the overall rate of new infections divided by the overall rate of recoveries, when $i \approx 0$, and $r = 0$ (then $s=1-i$). For \eqref{eq:classic_SIR}, this gives $\Ro = \beta/\gamma$. For our extended model \eqref{eq:extended_SIR_random}, we thus have
\[
\Ro
= \lim_{i\to 0} \frac{\mu (1-i) B(i)}{\gamma i}
= \frac{\mu B'(0)}{\gamma}.
\]
From \eqref{eq:Bi}, we see that $B'(i) = p \EE[\Theta (\Theta-1) (1-pi)^{\Theta-2}]$, and then
\[
\Ro
= \frac{\mu p \EE[\Theta(\Theta-1)]}{\gamma}.
\]

For instance:
\begin{itemize}
	\item when $\Theta \equiv \theta$ (model \eqref{eq:extended_SIR}), we have $\Ro = \mu p \theta(\theta-1) / \gamma$;
	
	\item if $\Theta \sim \text{binomial}(K,\alpha)$ (model \eqref{eq:extended_SIR_binomial}), then $\Ro = \mu p \alpha^2 K(K-1) / \gamma$;
	
	\item if $\Theta \sim \text{Poisson}(\lambda)$ (model \eqref{eq:extended_SIR_Poisson}), then $\Ro = \mu p \lambda^2 / \gamma$; and
	
	\item if $\Theta \sim \log(\alpha)$ (model \eqref{eq:extended_SIR_logarithmic}), then $\Ro = -\frac{\mu p \alpha^2}{\gamma (1-\alpha)^2 \log(1-\alpha)} = -\frac{\mu p\log(1-\alpha) \EE[\Theta]^2}{\gamma}$.
\end{itemize}

Notice that in these examples (and possibly many others) $\Ro$ grows quadratically with the average size of the gatherings. This fact may provide a new insight on the way that NPIs such as social distancing and lockdowns affect the spread of the disease. For instance, it may be a tool to determine how restriction policies on social gatherings should be defined; or it may be helpful to better understand why infection rates can rise significantly after lockdowns are lifted \cite{bruckhaus-martinez-garner-larocca-duncan2021}.

\subsection{Finite population stochastic dynamics and mean-field limit}
\label{sec:SIR_N}

The informal description of the classic SIR model \eqref{eq:classic_SIR} by means of
interactions between pairs of individuals in the infinite population, given in Section \ref{sec:classic_SIR}, can be made precise. To do so, one considers $N<\infty$ individuals subjected to these random encounters, giving rise to a jump Markov process on $\{0,\ldots,N\}^3$ corresponding to the evolution of the disease in the finite population. Compared to the completely deterministic system of ODEs, this stochastic process provides a more realistic representation of the complex random interactions that take place in a real-life epidemic. Nevertheless, despite the intrinsic randomness of this process, it can be shown that the proportions of susceptible, infected, and recovered individuals converge to the solution of \eqref{eq:classic_SIR}, in the limit as $N\to\infty$. This is a kind of law of large numbers, known as the \emph{mean-field limit} of the finite population dynamics. This type of result is essential, because it gives a full mathematical validation for the model. It was first proven in \cite{kurtz1970}, see also \cite[Chapter 11, Theorem 2.1]{ethier-kurtz} and \cite[Theorem 5.2]{andersson-britton2000}; for a more recent and elementary proof, see \cite{armbruster-beck2017}.

More generally, for our proposed model \eqref{eq:extended_SIR_random}, one can also study the finite population case. Let us thus consider the jump Markov process $X_t^N = (S_t^N, I_t^N, R_t^N)$ on $\{0,\ldots,N\}^3$ corresponding to the evolution of the number of susceptible, infected, and recovered individuals in the $N$-population, whose evolution consists of random social gatherings and recoveries, as described heuristically in Section \ref{sec:extended_SIR} for the case of an infinite population and $\Theta$ constant. Specifically, when $X_t^N$ is at state $(S,I,R)$, its dynamics is as follows:

\begin{itemize}
	\item At rate $\gamma I$, a recovery takes place: jump from $(S,I,R)$ to $(S,I-1,R+1)$.
	
	\item At rate $\mu N$, a gathering takes place:
	\begin{itemize}
		\item Sample the size of the gathering $\Theta$; if $\Theta>N$, nothing happens.
		
		\item Sample $\Theta$ individuals at random, without replacement, from the finite population; call $\tilde{S}^N, \tilde{I}^N, \tilde{R}^N$ the number of selected suceptibles, infected, and recovered (thus, $\tilde{S}^N + \tilde{I}^N + \tilde{R}^N = \Theta$).
		
		\item Each of the $\tilde{S}^N$ selected susceptibles will attempt to acquire the disease $\tilde{I}^N$ times, each time with probability $p$, independent of everything else. That is: if $J_k \in \{0, 1\}$ denotes the random variable that is equal to 1 when selected susceptible $k \in \{1,\ldots,\tilde{S}^N\}$ got infected, and 0 otherwise, then it is clear that $1-J_k \sim \text{Bernoulli}((1-p)^{\tilde{I}^N})$.
		
		\item Jump from $(S,I,R)$ to $(S-J,I+J,R)$, where $J = \sum_{k=1}^{\tilde{S}^N} J_k$ is the number of new infections.
		
	\end{itemize}
\end{itemize}

The process starts at $t=0$ from some given random vector $X_0^N = (S_0^N, I_0^N, R_0^N)$ satisfying $S_0^N + I_0^N + R_0^N = N$. This description unambiguously specifies the evolution of the process, which clearly satisfies $S_t^N + I_t^N + R_t^N = N$ for all $t\geq 0$; see also Section \ref{sec:LLN_general} for an explicit stochastic equation in terms of Poisson random processes. In Section \ref{sec:LLN_SIR_model_gatherings} we study the mean-field limit of $X_t^N$, i.e., we show that the proportions $S_t^N/N$, $I_t^N/N$, and $R_t^N/N$ converge, as $N\to\infty$, to the solution $(s_t,r_t,i_t)$ of \eqref{eq:extended_SIR_random}, just as in the classic setting \eqref{eq:classic_SIR}. This provides a completely rigorous mathematical justification for our model.

\subsection{Relation to previous works}

The present article extends the classic SIR model \eqref{eq:classic_SIR} of Kermack and McKendrick \cite{kermack-mckendrick1927} by introducing a non-linear dependence of the force of infection on $i$. This is not new: for instance, in \cite{liu-hethcote-levin1987,liu-levin-iwasa1986} the authors consider incidence rates of the kind $s^a i^b / (1+ci^{b-1})$ for some constants $a,b,c>0$, and study the dynamical behaviour of the associated system of ODEs. Similarly, in \cite{capasso-serio1978} the authors replace the term $\beta i$ in \eqref{eq:classic_SIR} by a general non-linear function satisfying some properties. One of the motivations of the authors was to capture the phenomenon of \emph{saturation} (see also \cite{kolokolnikov-iron2021}): as the number of infected individuals increases, the force of infection slows down, no longer increases linearly. Since this is certainly the case for \eqref{eq:extended_SIR_random} (unless $\Theta \in \{0,1,2\}$ a.s., see Proposition \ref{prop:properties_Bi}), the present paper can be seen as introducing a broad class of such non-linear functions $\mu B(i)$, by means of \eqref{eq:Bi}. Moreover, our model provides an intuitive explanation for the saturation phenomenon: social gatherings induce some redundancy in the infection mechanism. In other words: when the population has a significant number of infected individuals, some of their infectious power is lost because many gatherings will have few susceptibles available.

As mentioned in Section \ref{sec:random_size}, specific choices for the distribution of $\Theta$ yield formulas for the incidence rate already studied in the literature, of the form $s(1 - e^{-a i})$ \cite{kolokolnikov-iron2021} and $\frac{a s i}{1+bi}$ \cite{capasso-serio1978}. This shows that our proposed model is flexible and mathematically relevant, and at the same time it provides a probabilistic foundation for the use of those specific incidence rates.

Regarding the mean-field limit, it has been established for a model on $\RR^d$ much more general than \eqref{eq:classic_SIR}, see for instance \cite{andersson-britton2000,ethier-kurtz,kurtz1970}. In Section \ref{sec:LLN_general} we also state and prove a mean-field result ina a similar general setting, suitable to our purposes, see Theorem \ref{thm:LLN_general}. However, we remark that in the context of those references, the finite-population dynamics depends only on the proportions of each compartment, whereas in our proposed model \eqref{eq:extended_SIR_random} one needs to know the actual number of individuals in each compartment in order to sample the attendants of the gatherings. Consequently, the setting that we consider in Section \ref{sec:LLN_general} is slightly more general; specifically, the jump-rate functions will be allowed to depend on $N$. In Section \ref{sec:LLN_SIR_model_gatherings} we apply this to prove Theorem \ref{thm:LLN_extended_SIR}, which establishes the mean-field limit for \eqref{eq:extended_SIR_random}.

%

\section{Model derivation}
\label{sec:heuristic_justification}

We now proceed to deduce \eqref{eq:extended_SIR} and \eqref{eq:extended_SIR_random}. The population size is assumed to be infinite, which means that the developments of this section are still informal. We work with a fixed non-random size of gatherings $\theta \in \{0,1,2,\ldots\}$; the general case \eqref{eq:extended_SIR_random} follows just by taking expectations with respect to the randomness of $\Theta$.

Call $U$ the number of new infections in a gathering with $\theta$ individuals chosen randomly among the infinite population, where the proportion of susceptible, infected, and recovered individuals are $s$, $i$, and $r$, respectively. Since the average rate at which individuals change from susceptible to infected corresponds to $\mu \EE[U]$, to justify \eqref{eq:extended_SIR} we need to show that
\begin{equation}
\label{eq:EEU}
\EE[U] = \theta s \left( 1 - (1-ip)^{\theta-1} \right).
\end{equation}

Let $\tilde{S}$, $\tilde{I}$, and $\tilde{R}$ be the number of susceptible, infected, and recovered individuals, respectively, obtained after sampling $\theta$ individuals at random from the infinite population. Clearly, $(\tilde{S},\tilde{I},\tilde{R})$ has a multinomial distribution with probabilities $(s,i,r)$. Recall that each of the $\tilde{S}$ susceptibles will attempt to acquire the disease $\tilde{I}$ times, each time with probability $p$. Thus, given the value of $\tilde{I}$, the probability that a given susceptible in the room acquires the disease is $1-q^{\tilde{I}}$, where $q=1-p$. Therefore,
\begin{align}
\notag
\EE[U \mid \tilde{S},\tilde{I}]
&= \EE\left[ \sum_{k=1}^{\tilde{S}} \ind\{\text{susceptible $k$ got infected}\} ~\middle|~ \tilde{S},\tilde{I} \right] \\
\notag
&= \sum_{k=1}^{\tilde{S}} (1-q^{\tilde{I}}) \\
&= \tilde{S} - \tilde{S}q^{\tilde{I}}.
\label{eq:EEUmidW}
\end{align}
It is clear that $\tilde{S} \sim \text{binomial}(\theta,s)$ and $\tilde{I} \sim \text{binomial}(\theta,i)$. Moreover, given the value of $\tilde{I}$, we have $\tilde{S} \sim \text{binomial}(\theta - \tilde{I}, \frac{s}{s+r})$, which implies $\EE[\tilde{S} \mid \tilde{I}] = \frac{s}{s+r}(\theta-\tilde{I})$. Consequently, using \eqref{eq:Exi^Theta}, we have
\begin{align*}
\EE\left[ \EE[\tilde{S}q^{\tilde{I}} \mid \tilde{I} ] \right]
&= \frac{s}{s+r} \EE[ q^{\tilde{I}} (\theta - \tilde{I})] \\
&= \frac{s}{s+r} \left( \theta(1-i p)^\theta - \theta q i (1-ip)^{\theta-1} \right) \\
&= \theta s (1-ip)^{\theta-1}.
\end{align*}
Taking expectations in \eqref{eq:EEUmidW}, we thus obtain the desired expression \eqref{eq:EEU}:
\[
\EE[U]
= \theta s - \theta s (1-ip)^{\theta-1}
= \theta s \left( 1 - (1-ip)^{\theta-1} \right).
\]

\begin{remark}
In contrast, when the total population size is $N<\infty$, the $\theta$ individuals sampled (without replacement) from the finite population follow a \emph{multivariate hypergeometric distribution}. When studying the mean-field limit, one of the key points is a careful analysis on the convergence of this distribution, as $N\to\infty$, to the multinomial (sampling with replacement). This is performed in Step 4 of the proof of Theorem \ref{thm:LLN_extended_SIR}.
\end{remark}

\section{Main properties}
\label{sec:properties}

In this section we study the analytical behaviour of our proposed model \eqref{eq:extended_SIR_random}. The following proposition summarizes the main properties of the function $B(i)$ that we will use throughout this article. The proof is straightforward, so we omit it. To avoid trivial situations, in all what follows we will assume that $\Theta$, the random variable on $\{0,1,2,\ldots\}$ giving the size of the gatherings, satisfies
\[
\PP(\Theta \geq 2) >0
\quad \text{and} \quad
\EE[\Theta^2] < \infty.
\]

\begin{proposition}
	\label{prop:properties_Bi}
	Fix $p\in(0,1]$. Then, the function $B: [0,1] \to \RR_+$ given by \eqref{eq:Bi} satisfies the following properties:
	\begin{enumerate}
		
		\item $B(0) = 0$ and $B'(0) = p \EE[\Theta(\Theta-1)] > 0$.
		
		\item $B$ is strictly increasing.
		
		\item \label{prop:properties_Bi-iii}
		$B$ is Lipschitz continuous and concave. Moreover, $B(i) = B'(0) i$ if and only if $\Theta \in \{0,1,2\}$ a.s.; otherwise $B$ is strictly concave.
	\end{enumerate}
\end{proposition}

We now state the main properties of the solution of the system \eqref{eq:extended_SIR_random}. Notice that, thanks to the previous proposition, our setting is very similar to the one in \cite{capasso-serio1978}. Thus, the proof of the next lemma is a straightforward adaptation of the arguments in \cite[Section 4]{capasso-serio1978}, so we omit it here.

\begin{lemma}
	\label{lem:properties_sir}
	There exists a unique continuously differentiable solution $(s_t,i_t,r_t)_{t\geq 0}$ to \eqref{eq:extended_SIR_random}. Moreover, it satisfies the following properties:
	\begin{enumerate}
		\item $(s_t, i_t, r_t) \in [0,\infty)^3$ for all $t\geq 0$. If $s_0>0$, $i_0>0$, then $(s_t, i_t, r_t) \in (0,\infty)^3$ for all $t > 0$.
		
		\item $s_t + i_t + r_t = 1$ for all $t\geq 0$.
		
		\item $s_t$ is decreasing and $r_t$ is increasing. Moreover: if $s_0 \leq \gamma i_0 / (\mu B(i_0))$, then $i_t$ is decreasing; otherwise, $i_t$ first increases up to a maximum value, and then decreases.
		
		\item There exists $s_\infty \in [0,1]$ such that $\lim_{t\to\infty} (s_t,i_t,r_t) = (s_\infty, 0 ,1-s_\infty)$. If $s_0 \in (0,1)$, then $s_\infty \in (0,1)$ as well.
	\end{enumerate}
\end{lemma}

The next lemma tells us that the proportion of susceptibles in our proposed model \eqref{eq:extended_SIR_random} is always bounded below by the corresponding proportion for the classic SIR model \eqref{eq:classic_SIR}, provided that $\beta = \mu B'(0)$ and that both have the same initial conditions. In other words, assuming the same basic reproduction number, the disease spreads more slowly in our setting; in Figure \ref{fig:gSIR_classicSIR} we illustrate this fact numerically. This was already hinted in \cite{capasso-serio1978}; for convenience of the reader, here we provide a precise statement and proof.

\begin{lemma}
	\label{lem:comparison}
	Let $\beta = \mu B'(0)$. Denote $(s_t,i_t,r_t)$ and $(\hat{s}_t, \hat{\imath}_t, \hat{r}_t)$ the solutions to \eqref{eq:extended_SIR_random} and \eqref{eq:classic_SIR} respectively, with the same initial conditions $(s_0,i_0,r_0)$, with $s_0,i_0>0$. Write  $i = i(s)$ and $\hat{\imath} = \hat{\imath}(s)$ the proportions of infected individuals in terms of the variable $s \in (0,s_0]$ of the proportion of susceptibles. Then,
	\begin{align*}
	i(s) &\leq \hat{\imath}(s) \quad \text{for all $s\in(0,s_0]$, and} \\
	s_t &\geq \hat{s}_t,
	\quad \text{for all $t\geq 0$.}
	\end{align*}
\end{lemma}

\begin{proof}
	By Lemma \ref{lem:properties_sir}, we have $s_t,i_t,r_t>0$ for all $t>0$. We assume that $B$ is strictly concave; if not, then Proposition \ref{prop:properties_Bi}-\ref{prop:properties_Bi-iii} gives $B(i) = B'(0) i$, which implies that $i(s) \equiv \hat{\imath}(s)$ and $s_t \equiv \hat{s}_t$. From \eqref{eq:extended_SIR_random} and \eqref{eq:classic_SIR}, we obtain
	\[
	\frac{di}{ds} = -1 + \frac{\gamma i}{\mu s B(i)},
	\qquad \text{and} \qquad
	\frac{d\hat{\imath}}{ds} = -1 + \frac{\gamma}{\mu s B'(0)}.
	\]
	Consequently, for $u = u(s) = i(s) - \hat{\imath}(s)$, since $B(i) < B'(0)i$ for $i>0$, we have
	\[
	\frac{du}{ds} = \frac{\gamma}{\mu s} \left( \frac{i}{B(i)} - \frac{1}{B'(0)} \right)
	> 0.
	\]
	Since $u(s_0) = 0$, this gives $u(s) < 0$, that is, $i(s) < \hat{\imath}(s)$ for all $s \in (0,s_0)$.
	
	Now, let's go back to time variable. We argue by contradiction: assume that there exists $\bar{t}$ such that $s_{\bar{t}} < \hat{s}_{\bar{t}}$. Let $t_* \in [0, \bar{t})$ be the last time that $s_t$ and $\hat{s}_t$ were equal. We thus have
	\[
	s_{t_*}'
	= - \mu s_{t_*} B(i_{t_*})
	> - \mu \hat{s}_{t_*} B'(0) i_{t_*}
	> - \mu \hat{s}_{t_*} B'(0) \hat{\imath}_{t_*}
	= s_{t_*}',
	\]
	which contradicts the definition of $t_*$. Thus, we must have that $s_t \geq \hat{s}_t$ for all $t \geq 0$.
\end{proof}

\begin{figure}[t!]
	\includegraphics[width=\textwidth]{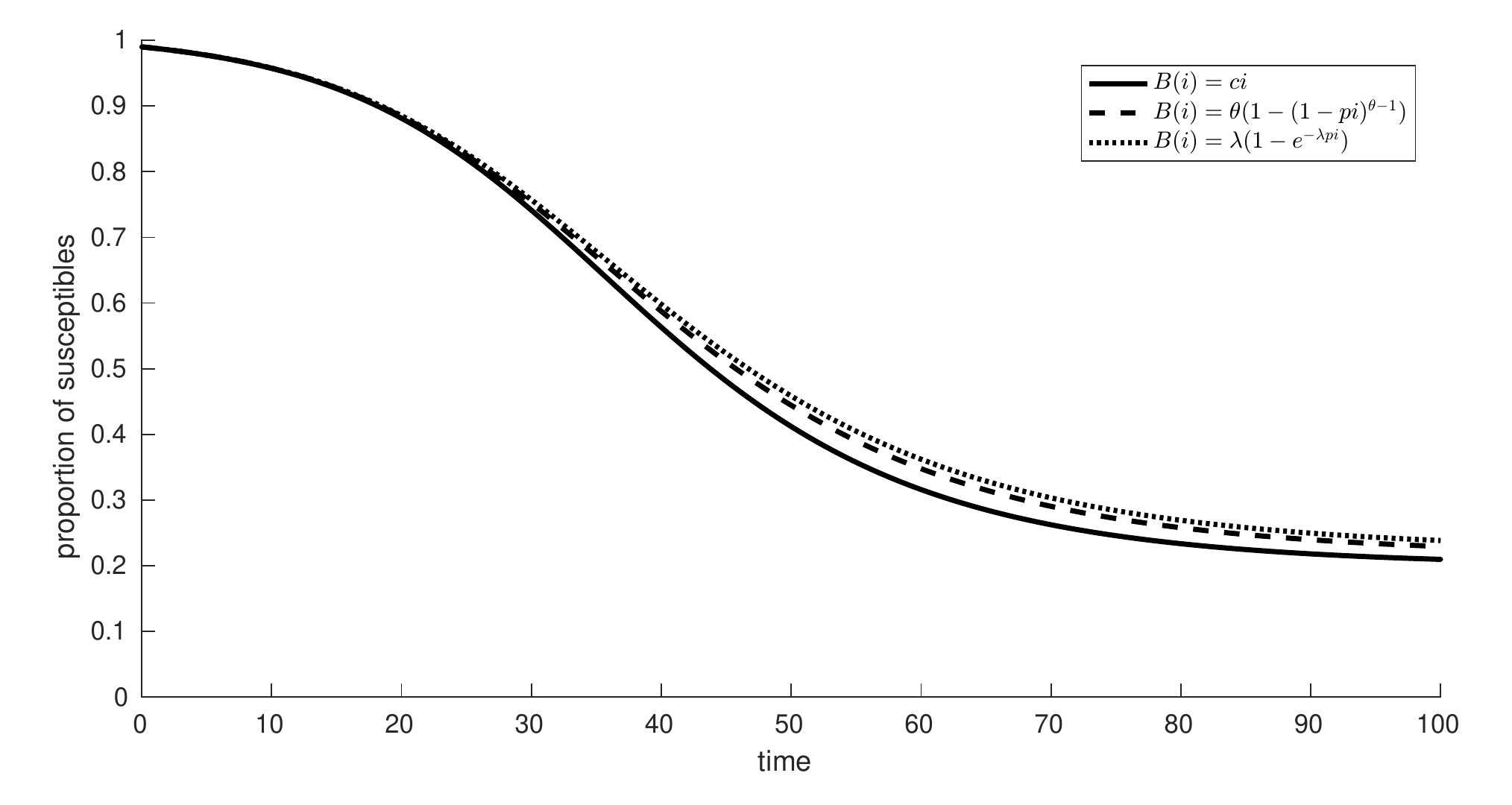}
	\caption{Evolution of the proportion of susceptible individuals $s$, obtained by solving numerically the general model \eqref{eq:extended_SIR_random} for three choices of the function $B(i)$. The parameters $c$, $\theta$ and $\lambda$ were chosen such that $B'(0)$ remains the same; specifically, $c = p\theta(\theta-1)$ and $\lambda^2 = \theta(\theta-1)$, for $\theta = 5$. Other parameters: $\mu = .05$, $p=.2$, $\gamma=.1$, $i_0 = .01$. Note that the curve corresponding to the classic SIR model \eqref{eq:classic_SIR}, i.e., $B(i) = ci$, remains below the other two curves.}
	\label{fig:gSIR_classicSIR}
\end{figure}

\section{Mean-field limit for a general model}
\label{sec:LLN_general}

Our goal now is to study the mean-field limit of the jump Markov process corresponding to the evolution of the disease in a finite population of size $N$ for our SIR model with gatherings \eqref{eq:extended_SIR_random}, as described in Section \ref{sec:SIR_N}. To that end, in this section we will first prove the desired convergence in a much more general setting; then, in Section \ref{sec:LLN_SIR_model_gatherings}, we will apply this result to \eqref{eq:extended_SIR_random}. We follow and generalize slightly the developments of \cite[Chapter 5]{andersson-britton2000}, see also \cite[Chapter 11]{ethier-kurtz}.

Let us describe the general model. Fix the total number of individuals $N \in \NN$. Let $(X_t^N)_{t\geq 0}$ be the jump Markov process on $\{0,\ldots,N\}^d$, starting from some given random initial condition $X_0^N \in \{0,\ldots,N\}^d$, such that for all states $x \in \{0,\ldots,N\}^d$ and jump amplitudes $\ell \in \ZZ^d$,
\begin{equation}
\label{eq:jump_rule}
\text{a jump from $x$ to $x+\ell$ occurs at rate $N \lambda_\ell^N(x)$.}
\end{equation}
Here, $\lambda_\ell^N : \{0,\ldots,N\}^d \to \RR_+$ is a rate function such that $\lambda_\ell^N(x) = 0$ whenever $x+\ell \notin \{0,\ldots,N\}^d$; this ensures that the process remains in $\{1,\ldots,N\}^d$.

\begin{example}
\label{ex:SIR_N}
For our SIR model with gatherings, we have $d=3$, and the rate functions $\lambda_\ell^N(\cdot)$ can be non-zero only for two types of jump amplitudes $\ell \in \ZZ^3$:
\begin{align*}
\ell_0 &:= (0,-1,1) & \text{(recovery)}, \\
\ell_k &:= (-k,k,0), \quad k=1,\ldots,N & \text{(infection of $k$ susceptibles)}.
\end{align*}
Specifically, for any $(S,I,R) \in \{0,\ldots,N\}^3$ with $S+I+R = N$,
\begin{align*}
\lambda_{\ell_0}^N(S,I,R)
&= \gamma \frac{I}{N}, \\
\lambda_{\ell_k}^N(S,I,R)
&= \mu \PP[U^N = k],
\end{align*}
where $U^N \in \{0,\ldots,S\}$ is the random variable of the number of newly infected individuals in a gathering of size $\Theta$ sampled randomly without replacement from the finite population $(S,I,R)$. Notice that, whereas one can write $\lambda_{\ell_0}^N(S,I,R) = \gamma i$ (dependence only on $i=I/N$), the functions $\lambda_{\ell_k}^N(S,I,R)$ for $k=1,\ldots,N$ do depend on $N$ through the distribution of $U^N$. This shows why the general model that we just introduced allows the rate functions to depend on $N$.
\end{example}

It is convenient to write the process $X_t^N$ explicitly. To that end, consider a collection $(\P_\ell(t))_{\ell \in \ZZ^d}$ of independent Poisson processes on $\RR$ with intensity 1. It is straightforward to see that the stochastic equation
\begin{equation}
\label{eq:XtN}
X_t^N = X_0^N + \sum_{\ell \in \ZZ^d} \ell \P_\ell\left(N  \int_0^t \lambda_\ell^N(X_s^N) ds \right)
\end{equation}
indeed defines a process with jump intensities given by \eqref{eq:jump_rule}.

In Theorem \ref{thm:LLN_general} below we present a convergence result for the process $Z_t^N = \frac{1}{N} X_t^N$, which is a more general version of \cite[Theorem 5.2]{andersson-britton2000}. In that reference, the proof was based on the following well known fact: for a Poisson process $\P(t)$ on $\RR$ with intensity $\lambda$, one has for all $T\geq 0$:
\begin{equation}
\label{eq:Poisson_LLN_1D}
\lim_{N\to\infty} \sup_{t \leq T} \left| \frac{1}{N}\P(Nt) - \lambda t \right| = 0
\quad \text{a.s.}
\end{equation}
In the present paper, because our setting is more general (in particular, the jump rate functions $\lambda_\ell^N(x)$ are not required to be identically 0 for all but finitely many $\ell$'s, uniformly on $N$), we will need the following stronger version. The proof is somewhat technical, so it may be skipped at first reading. In what follows, $|\cdot|$ denotes the 1-norm on $\RR^d$.

\begin{lemma}
\label{lem:Q_LLN}
Consider a collection of non-negative numbers $(\bar{\lambda}_\ell)_{\ell \in \ZZ^d}$ such that $\sum_{\ell} |\ell| \bar{\lambda}_\ell < \infty$, and define the process on $\ZZ^d$
\[
\Q(t) = \sum_{\ell \in \ZZ^d} \ell \P_\ell( \bar{\lambda}_\ell t).
\]
Then $\Q$ is well defined. Moreover, for $\bar{\lambda} = \sum_{\ell} \ell \bar{\lambda}_\ell \in \RR^d$, it satisfies for all $T\geq 0$:
\[
\lim_{N\to\infty} \sup_{t \leq T} \left| \frac{1}{N}\Q(Nt) - t \bar{\lambda} \right| = 0
\quad \text{a.s.}
\]	
\end{lemma}

\begin{proof}
For simplicity, we work with $d=1$; the general case is obtained arguing component-wise.

Note that $| \EE[\Q(t)] | \leq \sum_\ell |\ell| \bar{\lambda}_\ell t < \infty$, thus $\Q$ is well defined. To prove the desired convergence, the main idea of the argument, which can be found in \cite[Chapter 11, Theorem 2.1]{ethier-kurtz}, is the following:
\begin{align*}
\lim_{N\to\infty} \sup_{t \leq T} \left| \frac{1}{N}\Q(Nt) - t \bar{\lambda} \right|
&\leq \lim_{N\to\infty} \sum_{\ell \in \ZZ} |\ell| \sup_{t \leq T} \left| \frac{1}{N} \P_\ell(\bar{\lambda}_\ell N t) - \bar{\lambda}_\ell t \right| \\
&= \sum_{\ell \in \ZZ} |\ell| \lim_{N\to\infty}  \sup_{t \leq T} \left| \frac{1}{N} \P_\ell(\bar{\lambda}_\ell N t) - \bar{\lambda}_\ell t \right|,
\end{align*}
which equals 0 thanks to \eqref{eq:Poisson_LLN_1D}. However, it is not obvious that one can exchange the limit and the summation; we will spend the rest of the proof justifying this step. We will use the following ``converse'' of the dominated convergence theorem, which can be found in \cite{rennie1960}:

\begin{proposition}
\label{prop:DCT_converse}
Let $f_N$, $f$ be integrable functions on a $\sigma$-finte measure space $(E,\nu)$ such that $\lim_N f_N = f$ $\nu$-a.s., and $\lim_N \int f_N g d\nu = \int f g d\nu$
for all $g$ bounded and measurable. Then, any subsequence of $(f_N)_{N\in\NN}$ has a sub-subsequence which is dominated by an integrable function.
\end{proposition}

Now, let $h_N(\ell)
= |\ell| \sup_{t \leq T} | \frac{1}{N} \P_\ell(\bar{\lambda}_\ell N t) - \bar{\lambda}_\ell t|
$. We want to show that $\lim_N \sum_\ell h_N(\ell) = 0$. We argue by contradiction: assume that, modulo subsequence, we have $\lim_N \sum_\ell h_N(\ell) =: a > 0$. Note that
\begin{equation}
\label{eq:hNell}
h_N(\ell)
\leq |\ell| \frac{1}{N} \P_\ell(\bar{\lambda}_\ell N T) + |\ell| \bar{\lambda}_\ell T
=: f_N(\ell).
\end{equation}
We will apply Proposition \ref{prop:DCT_converse} in the measure space $E=\ZZ$ with $\nu$ being the counting measure. Fix $g : \ZZ \to \RR$ bounded. Then:
\begin{align*}
\int f_N g d\nu
&= \sum_{\ell \in \ZZ} |\ell| \frac{1}{N} \P_\ell(\bar{\lambda}_\ell N T) g(\ell)
+ T \sum_{\ell \in \ZZ} |\ell| \bar{\lambda}_\ell g(\ell) \\
&= \frac{1}{N} \sum_{k=1}^N Y_k
+ T \sum_{\ell \in \ZZ} |\ell| \bar{\lambda}_\ell g(\ell),
\end{align*}
where
\[
Y_k := \sum_{\ell \in \ZZ} |\ell| g(\ell) \left\{ \P_\ell(k \bar{\lambda}_\ell T) - \P_\ell((k-1)\bar{\lambda}_\ell T) \right\}.
\]
Since the $\P_\ell$'s have independent increments, we see that $(Y_k)_{k\in \NN}$ are i.i.d., with $\EE[Y_k] = T \sum_\ell |\ell| \bar{\lambda}_\ell g(\ell)  < \infty$. Thus, by the strong law of large numbers, we deduce that $\frac{1}{N} \sum_{k=1}^N Y_k$ converges to $T \sum_\ell |\ell| \bar{\lambda}_\ell g(\ell)$, $\PP$-a.s. Consequently, $\lim_N \int f_N g d\nu = \int f g d\nu$, for the $\nu$-integrable function $f(\ell) := 2T |\ell| \bar{\lambda}_\ell$. Similarly, for all $\ell$ we have $\lim_N f_N(\ell) = f(\ell)$, $\PP$-a.s.

We can thus apply Proposition \ref{prop:DCT_converse} and deduce that, modulo subsequence, $f_N$ is dominated by some $\nu$-integrable function. By \eqref{eq:hNell}, this implies that $h_N$ is also dominated. The dominated convergence theorem now gives
\[
0
< a
= \sum_{\ell \in \ZZ} \lim_{N \to \infty} h_N(\ell)
= 0,
\]
thanks to \eqref{eq:Poisson_LLN_1D}. This is a contradiction, which concludes the proof.
\end{proof}

In order to state and prove our general result, rather than dealing directly with $\lambda_\ell^N(x)$, which gives the rate for each jump amplitude $\ell \in \ZZ^d$, it is convenient to study the \emph{expected} jump rate amplitudes: define the function
\begin{equation}
\label{eq:FN}
F^N(z) = \sum_{\ell \in \ZZ^d} \ell \lambda_\ell^N(Nz),
\qquad
\forall z \in \left\{0,\frac{1}{N},\ldots,1 \right\}^d.
\end{equation}
That is, when the process is at state $Nz$, the vector $F^N(z)$ corresponds to the expected jump rate amplitudes. We will require that $F^N$ converges in some sense to a function $F : [0,1]^d \to \RR^d$, as $N\to\infty$. We will quantify this using the uniform norm:
\begin{equation}
\label{eq:infty_norm}
\Vert F^N - F \Vert_\infty
:= \sup_{z \in \{0,\frac{1}{N},\ldots,1 \}^d} |F^N(z) - F(z)|.
\end{equation}
Moreover, given some $z_0 \in [0,1]^d$, we will typically denote $(z_t)_{t\geq 0}$ the solution to the differential equation in integral form
\begin{equation}
\label{eq:zt}
z_t = z_0 + \int_0^t F(z_s) ds,
\qquad \forall t\geq 0.
\end{equation}

\begin{theorem}
	\label{thm:LLN_general}
	Let $(X_t)_{t\geq 0}$ be the jump Markov process on $\{0,\ldots,N\}^d$ given by the rule \eqref{eq:jump_rule}, and denote $Z_t^N = \frac{1}{N}X_t^N$. Assume that there exist a collection of non-negative numbers $(\bar{\lambda}_\ell)_{\ell \in \ZZ^d}$ and a function $F: [0,1]^d \to \RR^d$ such that:
	\begin{enumerate}
		\item \label{thm:LLN_general-i} $\lambda_\ell^N(x) \leq \bar{\lambda}_\ell$ for all $N\in \NN$ and $x \in \{0,\ldots,N\}^d$, and $\sum_\ell |\ell| \bar{\lambda}_\ell < \infty$,
		
		\item \label{thm:LLN_general-ii} $F$ is Lispchitz, and $\lim_N \Vert F^N - F \Vert_\infty =0 $,
	\end{enumerate}
	Assume also that $\lim_N Z_0^N = z_0$ a.s.\ for some $z_0 \in [0,1]^d$, and let $(z_t)_{t\geq 0}$ be the solution to \eqref{eq:zt} associated with $F$. Then, for all $T\geq 0$,
	\[
	\lim_{N\to\infty} \sup_{t \leq T} |Z_t^N - z_t| = 0
	\quad \text{a.s.}
	\]
\end{theorem}

\begin{proof}
The proof is an extension of \cite[Theorem 5.2]{andersson-britton2000}. Call $\hat{\P}_\ell(t) = \P_\ell(t) -t$. From \eqref{eq:XtN}, dividing by $N$, gives
\[
Z_t^N
= Z_0^N
+ \frac{1}{N} \sum_{\ell \in \ZZ^d} \ell \hat{\P}_\ell\left( N \int_0^t \lambda_\ell^N(N Z_s^N) ds \right)
+ \int_0^t F^N(Z_s^N) ds.
\]
From \eqref{eq:zt}, we have for all $t\leq T$:
\begin{align*}
& |Z_t^N - z_t| \\
& \leq |Z_0^N - z_0|
+ \sup_{s\leq t} \left| \frac{1}{N} \sum_{\ell \in \ZZ^d} \ell \hat{\P}_\ell\left( N \int_0^s \lambda_\ell^N(N Z_u^N) du \right) \right|
+ \int_0^t |F^N(Z_s^N) - F(z_s)| ds \\
& \leq |Z_0^N - z_0|
+ \sup_{s\leq t} \left| \frac{1}{N} \sum_{\ell \in \ZZ^d} \ell \hat{\P}_\ell( N \bar{\lambda}_\ell s) \right| + t \Vert F^N - F \Vert_\infty
+ L \int_0^t  |Z_s^N - z_s| ds,
\end{align*}
where in the second term we have used that $\lambda_\ell^N(\cdot) \leq \bar{\lambda}_\ell$, and in the third term we added and subtracted $F(Z_s^N)$ and used that $F$ is $L$-Lipschitz. Using Grönwall's lemma, we thus obtain
\[
|Z_t^N - z_t|
\leq \left( |Z_0^N - z_0|
+ \sup_{s\leq t} \left| \frac{1}{N} \sum_{\ell \in \ZZ^d} \ell \hat{\P}_\ell( N \bar{\lambda}_\ell s) \right|
+ t \Vert F^N - F \Vert_\infty \right) e^{L t}.
\]
Thus, $\sup_{t\leq T} |Z_t^N - z_t|$ is bounded by the right hand side with $t$ replaced by $T$. The first and third term converge to 0 as $N\to\infty$ by assumption, whereas the second term also converges to 0, thanks to Lemma \ref{lem:Q_LLN}. The result follows taking limits.
\end{proof}

\section{Mean-field limit for the SIR model with gatherings}
\label{sec:LLN_SIR_model_gatherings}

Finally, we use Theorem \ref{thm:LLN_general} to state and prove the following result, which establishes the mean-field limit of the finite population Markov process associated with our proposed model \eqref{eq:extended_SIR_random}.

\begin{theorem}
	\label{thm:LLN_extended_SIR}
	Let $X_t^N$ be the Markov process on $\{0,\ldots,N\}^3$ described in Section \ref{sec:SIR_N}, and also in Example \ref{ex:SIR_N}. Denote $Z_t^N = \frac{1}{N} X_t^N$, and let $z_t = (s_t,i_t,r_t)$ be the solution to \eqref{eq:extended_SIR_random}. Assume that $\lim_N Z_0^N = z_0$. Then, for all $T\geq 0$,
	\[
	\lim_{N\to \infty} \sup_{t\leq T} |Z_t^N - z_t|
	= 0
	\quad \text{a.s.}
	\]
\end{theorem}

\begin{proof}
	Let us fix some notation. Denote $W^N = (\tilde{S}^N, \tilde{I}^N, \tilde{R}^N)$ (respectively, $W = (\tilde{S}, \tilde{I}, \tilde{R})$) the number of selected susceptible, infected, and recovered individuals in a gathering, drawn at random, without replacement, from a finite population with sizes $(S,I,R) \in \{0,\ldots,N\}^3$, $S+I+R=N$ (respectively, infinite population with proportions $(s,i,r) = (\frac{S}{N}, \frac{I}{N}, \frac{R}{N})$). Given the value of $\Theta$, it is clear that $W^N$ has multivariate hypergeometric distribution, whereas $W$ is multinomial. Also, $U^N$ (respectively, $U$) denotes the number of new infections in the gathering in the finite case (respectively, infinite).
	
	We will apply Theorem \ref{thm:LLN_general}, for which we need to check that conditions \ref{thm:LLN_general-i} and \ref{thm:LLN_general-ii} are satisfied. We adopt the general notation of Section \ref{sec:LLN_general}.
	We split the proof in steps.
	
	\medskip
	\textit{Step 1.}
	We first check  \ref{thm:LLN_general-i}. With the notation of Example \ref{ex:SIR_N}, we only need to find uniform (in $N$) bounds for the rate functions $\lambda_{\ell_0}^N$ and $\lambda_{\ell_k}^N$, $k=1,\ldots,N$. For the former, we have
	\[
	\lambda_{\ell_0}^N(S,I,R)
	= \gamma \frac{I}{N}
	\leq \gamma
	=: \bar{\lambda}_{\ell_0},
	\]
	whereas for $k=1,\ldots,N$,
	\begin{align*}
	\lambda_{\ell_k}^N(S,I,R)
	&= \mu \PP[U^N = k] \\
	&= \mu \sum_{\theta = k}^N \PP[U^N = k \mid \Theta = \theta] \PP[\Theta = \theta] \\
	&\leq \mu \PP[\Theta \geq k] \\
	&=: \bar{\lambda}_{\ell_k}.
	\end{align*}
	Consequently:
	\begin{align*}
	\sum_{\ell \in \ZZ^d} |\ell| \bar{\lambda}_\ell^N
	&= |\ell_0| \gamma + \mu \sum_{k=1}^\infty |\ell_k| \PP[\Theta \geq k] \\
	&\leq 2 \gamma + 2 \mu \sum_{k=1}^\infty k \PP[\Theta \geq k] \\
	&= 2 \gamma + 2 \mu \EE\left[ \frac{\Theta(\Theta+1)}{2} \right],
	\end{align*}
	which is finite thanks to the assumption $\EE[\Theta^2] < \infty$. Thus, condition \ref{thm:LLN_general-i} is satisfied.
	
	\medskip
	\textit{Step 2.}
	We now aim to check \ref{thm:LLN_general-ii}. From Example \ref{ex:SIR_N} and the definition of $F^N$ \eqref{eq:FN}, we see that
	\[
	F^N(s,i,r)
	= (-\mu \EE[U^N], \mu \EE[U^N] - \gamma i, \gamma i),
	\]
	and the natural candidate for the limiting function $F$ is
	\begin{align*}
	F(s,i,r)
	&= (-\mu \EE[U], \mu \EE[U] - \gamma i, \gamma i) \\
	&= (-\mu s B(i), \mu s B(i) - \gamma i, \gamma i),
	\end{align*}
	where the last equality was verified in Section \ref{sec:heuristic_justification}. Therefore, \eqref{eq:zt} in the present setting is just \eqref{eq:extended_SIR_random} written in integral form. As in \eqref{eq:EEUmidW}, we have
	\begin{align*}
	\EE[U^N]
	&= \EE[\tilde{S}^N] - \EE[\tilde{S}^N q^{\tilde{I}^N}], \\
	\EE[U]
	&= \EE[\tilde{S}] - \EE[\tilde{S} q^{\tilde{I}}],
	\end{align*}
	where $q=1-p$. Thus,
	\begin{equation}
	\label{eq:FN-F}
	|F^N(s,i,r) - F(s,i,r)|
	\leq 2\mu \left|\EE[\tilde{S}^N] - \EE[\tilde{S}]\right| + 2\mu \left|\EE[\tilde{S}^N q^{\tilde{I}^N}] - \EE[\tilde{S} q^{\tilde{I}}]\right|.
	\end{equation}
	Consequently, to check condition \ref{thm:LLN_general-ii}, it suffices to show that those two terms converge to 0 as $N\to\infty$ uniformly on $(s,i,r)$.
	
	\medskip
	\textit{Step 3.} We start with the first term of \eqref{eq:FN-F}. Clearly, given $\Theta = \theta \leq N$, we have $\tilde{S}^N \sim \text{hypergeom}(N,S,\theta)$, thus
	\[
	\EE[\tilde{S}^N]
	= \sum_{\theta=1}^N \EE[\tilde{S}^N \mid \Theta = \theta]\PP[\Theta=\theta]
	= \sum_{\theta=1}^N \frac{S \theta}{N} \PP[\Theta=\theta]
	= s \EE[\Theta \ind\{\Theta \leq N\}],
	\]
	and since $\EE[\tilde{S}] = s \EE[\Theta]$, we obtain
	\begin{equation}
	\label{eq:ESN-ES}
	\left|\EE[\tilde{S}^N] - \EE[\tilde{S}]\right|
	\leq \EE[\Theta \ind\{\Theta > N\}].
	\end{equation}
	
	\medskip
	\textit{Step 4.}
	We now study the second term in  \eqref{eq:FN-F}; this is the key part of the proof. Denote $\nu^N = \law(W^N)$, $\nu = \law(W)$. Since the function $\phi(x,y,z) = x q^y$ is bounded, we have
	\begin{equation}
	\label{eq:ESNqIN-ESqI}
	\left|\EE[\tilde{S}^N q^{\tilde{I}^N}] - \EE[\tilde{S} q^{\tilde{I}}]\right|
	\leq \Vert \phi \Vert_{\infty} \Vert \nu^N - \nu \Vert_{\TV},
	\end{equation}
	where $\Vert \cdot \Vert_{\TV}$ denotes the total variation norm. It is defined as
	\[
	\Vert \nu^N - \nu \Vert_{\TV}
	= 2 \inf \PP[W^N \neq W],
	\]
	where the infimum is taken over all couplings, that is, over all possible ways of defining $W^N$ and $W$ on a common probability space, with $W^N \sim \nu^N$ and $W \sim \nu$. We will now define one particular coupling: sample $\Theta$ as usual, and then draw $\Theta$ elements from the set $\{1,\ldots,N\}$ \emph{with replacement}; now, define $W = (\tilde{S}, \tilde{I}, \tilde{R})$ as
	\begin{align*}
	\tilde{S} &= \text{number of elements in $\{1,\ldots,S\}$}, \\
	\tilde{I} &= \text{number of elements in $\{S+1,\ldots,S+I\}$}, \\
	\tilde{R} &= \text{number of elements in $\{S+I+1,\ldots,N\}$},
	\end{align*}
	and
	\[
	W^N
	= \begin{cases}
	W & \text{if $\Theta\leq N$ and there were no repeated elements}, \\
	\hat{W}^N & \text{if $\Theta\leq N$ and some element was repeated}, \\
	(0,0,0) & \text{if $\Theta>N$},
	\end{cases}
	\]
	where $\hat{W}^N$ is some independent realization of $\nu^N$. It is clear that $W^N \sim \nu^N$ and $W \sim \nu$. Moreover:
	\begin{align*}
	\PP[W^N = W]
	&= \sum_{\theta=1}^N \PP[W^N = W \mid \Theta = \theta] \PP[\Theta=\theta] \\
	&\geq \sum_{\theta=1}^N \PP[\text{no repeated elements} \mid \Theta = \theta] \PP[\Theta=\theta] \\
	&= \sum_{\theta=1}^N \frac{N}{N} \frac{N-1}{N} \cdots \frac{N-\theta+1}{N}  \PP[\Theta=\theta] \\
	&= \EE\left[\frac{N}{N} \frac{N-1}{N} \cdots \frac{N-\Theta+1}{N} \right],
	\end{align*}
	thus
	\begin{equation}
	\label{eq:TVnuN-nu}
	\frac{1}{2} \Vert \nu^N - \nu \Vert_{\TV}
	\leq 1 - \PP[W^N = W]
	\leq 1 - \EE\left[\frac{N}{N} \frac{N-1}{N} \cdots \frac{N-\Theta+1}{N} \right].
	\end{equation}
	
	\medskip
	\textit{Step 5.} Finally, from \eqref{eq:FN-F}, \eqref{eq:ESN-ES}, \eqref{eq:ESNqIN-ESqI} and \eqref{eq:TVnuN-nu}, we obtain
	\begin{align*}
	&|F^N(s,i,r) - F(s,i,r)| \\
	&\leq 2\mu \EE[\Theta \ind\{\Theta > N\}]
	+ 4\mu \Vert \phi \Vert_\infty \left\{ 1 - \EE\left[\frac{N}{N} \frac{N-1}{N} \cdots \frac{N-\Theta+1}{N} \right] \right\}.
	\end{align*}
	Both terms go to 0 as $N\to\infty$, the first one by monotone convergence and the second one by dominated convergence. Noting that this does not depend on $(s,i,r)$, taking supremum over $(s,i,r) \in \{0,\frac{1}{N},\ldots,1\}^3$, shows that $\lim_N \Vert F^N - F \Vert_\infty = 0$ (the norm $\Vert \cdot \Vert_\infty$ was defined in \eqref{eq:infty_norm}). Thus, \ref{thm:LLN_general-ii} is checked and the proof is complete.
\end{proof}

\bibliographystyle{plain}
\bibliography{references.bib}{}

\end{document}